\documentclass[final]{siamltex}

\usepackage{graphicx}
\usepackage{enumitem}

\usepackage{color}
\usepackage{bm}
\usepackage{amsfonts}
\usepackage{amsmath}

\usepackage[tight,footnotesize]{subfigure}

\newcommand{\C}[1]{{\cal {#1}}}

\renewcommand{\bar}{\overline}

\title{Second-order sufficient optimality conditions in the
calculus of variations
\thanks{
November 27, 2024.
The author gratefully acknowledges support by the National
Science Foundation under grant 2031213.}}
\author{
    William W. Hager\thanks{{\tt hager@ufl.edu},
        http://people.clas.ufl.edu/hager/,
        PO Box 118105,
        Department of Mathematics,
        University of Florida, Gainesville, FL 32611-8105.
        Phone (352) 294-2308.}
}
\begin{document}
\maketitle
\begin{abstract}
Some classic second-order sufficient optimality conditions in the calculus of
variations are shown to be equivalent, while also introducing a new
equivalent second-order condition which is extremely easy to apply:
simply integrate a linear second-order initial value problem
and check that the solution is positive over the problem domain.
\end{abstract}


\section{Introduction}
This note concerns an efficient algorithm for showing that a given
stationary point in the calculus of variations is a strict local minimizer.
Our focus is second-order sufficient optimality conditions for
classic problems in the calculus of variations such as
\begin{equation}\label{P}
\min ~ F(y) := \int_0^1 f(y'(x), y(x), x) ~dx \quad \mbox{subject to }
y \in \C{C}_0^1,
\end{equation}
where $\C{C}_0^1$ denotes the space of continuously differentiable,
real-valued functions defined on $[0, 1]$ with $y(0) = y(1) = 0$.
If $f \in \C{C}^1$, then at a local minimizer $y^*$, Euler's equation
is satisfied:
\begin{equation}\label{E}
\frac{d}{dx} \left( \frac{\partial f}{\partial y'}
\left(({y^*}'(x), y^*(x), x) \right) \right) =
\frac{\partial f}{\partial y} \left(({y^*}'(x), y^*(x), x) \right) .
\end{equation}
This is a necessary optimality condition.

A trivial second-order sufficient optimality condition is formulated
in terms of the Hessian of the integrand of $f$ along a solution $y^*$ of
(\ref{E}).
Suppose that $f \in \C{C}^2$ and $y^*$ satisfies the Euler equation (\ref{E}).
Let $z \in \mathbb{R}^2$ and let us consider the $2$ by $2$ Hessian
\begin{equation}\label{Q}
\nabla^2_{z} f\left( z, x \right)
\bigg|
_{z = ({y^*}'(x), y^*(x))} .
\end{equation}
If this Hessian is positive definite for all $x \in [0, 1]$, then $y^*$ is
a strict local minimizer of (\ref{P}).
This follows from a second-order Taylor expansion of the integrand
of (\ref{P}) around $({y^*}'(x), y^*(x))$.
The first-order term in the Taylor expansion vanishes since $y^*$ satisfies
the Euler equation (\ref{E}), while the second-order term in the expansion
is nonnegative in a neighborhood of $({y^*}'(x), y^*(x))$.
The requirement of positive definiteness of the Hessian (\ref{Q})
for all $x \in [0, 1]$ is a rather strong condition which limits the
applicability of this second-order condition.

A second-order sufficient optimality condition that does not require
positive definiteness of the quadratic (\ref{Q}) on $[0, 1]$
was obtained by Legendre.
Let $\Omega$ denote the quadratic form associated with the Hessian (\ref{Q}):
\[
\Omega(h) = \int_0^1 P(x) h'(x)^2 + 2R(x) h'(x)h(x) + Q(x) h(x)^2 \; dx,
\]
where $h \in \C{C}_0^1$ and
$P$, $Q$, and $R$ correspond to the $(1, 1)$, $(2,2)$, and $(1,2)$
elements respectively of the Hessian (\ref{Q}).
If $f \in \C{C}^2$, then $P$, $Q$, and $R \in \C{C}^0$.
Any conditions that ensure a lower bound of the form
\[
\Omega(h) \ge \gamma \int_0^1 h'(x)^2 + h(x)^2 \; dx
\]
for some $\gamma > 0$ lead to a sufficient optimality condition for (\ref{P}).
Since
\[
\frac{d}{dx} h(x)^2 = 2 h(x)h'(x),
\]
an integration by parts yields
\begin{equation} \label{compressedOmega}
\Omega (h) = \int_0^1 P(x) h'(x)^2 + [Q(x) - R'(x)]h(x)^2 \; dx
\end{equation}
assuming $R \in \C{C}^1$.
Hence, when $R \in \C{C}^1$, there is no loss of generality in focusing on
quadratic forms with $R = 0$.
Let $\Gamma(h)$ denote the quadratic form $\Omega(h)$ in the special case
where $R = 0$:
\begin{equation}\label{Gamma-def}
\Gamma(h) = \int_0^1 P(x) h'(x)^2 + Q(x)h(x)^2 \; dx,
\end{equation}
where the $Q$ term in (\ref{Gamma-def}) is identified with the
$Q - R'$ term in (\ref{compressedOmega}).

If $\Omega(h) \ge 0$ for all $h \in \C{C}_0^1$, then
$P \ge 0$ on $[0, 1]$;
consequently, Legendre focused on quadratic forms $\Omega$ where
$P(x) \ge \alpha > 0$ for all $x \in [0, 1]$, and he completed
the square in $\Gamma$ using any bounded solution on $[0, 1]$ to the
Ricatti equation
\begin{equation}\label{R}
w'(x) = \frac{w(x)^2}{P(x)} - Q(x) .
\end{equation}
Adding the identity
\[
0 = \int_0^1 \frac{d}{dx}\left( w(x) h(x)^2 \right) =
\int_0^1 \left[ w'(x) h(x)^2  + 2w(x)h(x)h'(x) \right] \; dx
\]
to $\Gamma$ and substituting for $w'$ using the Ricatti equation yields a
perfect square:
\begin{equation}\label{S}
\Gamma(h) =
\int_0^1 P(x) \left( h'(x) + \frac{w(x)}{P(x)} h(x) \right)^2 \; dx.
\end{equation}
Since $P$ is strictly positive,
$\Gamma (h) \ge 0$ and $\Gamma (h) = 0$ if and only if $h = 0$,
or equivalently, the quadratic form $\Gamma$ is positive definite.
In the next section, we show that $\Gamma$ is positive definite
if and only if there exists a bounded solution of the
Ricatti equation (\ref{R}) on $[0, 1]$.

A different, but related view of second-order optimality,
is developed by Gelfand and Fomin \cite{GelfandFomin}
in the context of conjugate points.
The Euler equation associated with a stationary point of $\Gamma$ is
\begin{equation}\label{EE}
\frac{d}{dx} (P(x) u'(x)) = Q(x) u(x) .
\end{equation}
It is assumed that $P \in \C{C}^1$ and there exists $\alpha$ such
that $P(x) \ge \alpha > 0$ for all $x \in [0, 1]$.
Gelfand and Fomin consider (\ref{EE}) in the context of a boundary-value
problem with conditions
\begin{equation}\label{BVP}
u(0) = u(a) = 0
\end{equation}
where $a \in (0, 1]$.
The point $x = a$ is said to be conjugate to $x = 0$ if
there exists a nonzero solution to (\ref{EE}) satisfying (\ref{BVP}).
On pages 106--111 of \cite{GelfandFomin}, Gelfand and Fomin show that
$\Gamma$ is positive definite over $C_0^1$ if and only if
$a$ is never conjugate to $x = 0$ for any choice of $a \in (0, 1]$.

We will also view (\ref{EE}) as an initial-value problem with initial conditions
\begin{equation}\label{IVP}
u(0) = u_0 \quad \mbox{and} \quad u'(0) = u_1.
\end{equation}
Since (\ref{EE}) is a linear equation,
there exists a unique solution on $[0, 1]$ to the initial-value problem
(\ref{IVP}) for any choice of $u_0$ and $u_1$.

The development of second-order conditions in optimization is an area of
intense research that continues today.
As an example of current research, the paper \cite{Frankowska24} of
Frankowska and Osmolovskii develops second-order optimality conditions for a
general set in a Banach space.
There has been particular interest in the development of second-order
conditions in optimal control;
a few references include \cite{casas08, osmolovskii12, Vossen2010}.
However, here our interest is not in the further development or extension of
second-order conditions, but rather in numerically checking whether
a second-order condition is satisfied.
In general, it may not be easy to use either the Ricatti equation
or the conjugacy conditions to check numerically whether a solution
of Euler's equation (\ref{E}) is a strict local minimizer.
With the Ricatti equation strategy, showing that the second-order
optimality condition holds amounts to finding an initial condition for which
the solution to the Ricatti equation (\ref{R}) does not blow up on $[0, 1]$.
With the conjugate point approach, showing that the second-order
optimality condition holds amounts to verifying
that there does not exist $a \in (0, 1]$ that is conjugate to $x = 0$.
The goal of this paper is to establish the close connection between both these
approaches to second-order optimality, while also providing a simple numerical
test for a strict local minimizer.

\section{Main Results}

Our observations are contained in the following theorem.

\begin{theorem} \label{maintheorem}
Suppose that there exists $\alpha$ such that $P(x) \ge \alpha > 0$ for
all $x \in [0, 1]$, where $P \in \C{C}^1$ and $Q \in \C{C}^0$ in
$(\ref{Gamma-def})$.
The following conditions are equivalent:
\begin{enumerate}
\item[{\rm (C1)}]
There exists a strictly positive solution $u$ of the linear differential
equation {\rm (\ref{EE})} on $[0, 1]$.
\item[{\rm (C2)}]
There exists a uniformly bounded solution $w$ of the Ricatti equation
{\rm (\ref{R})} on $[0, 1]$.
\item[{\rm (C3)}]
There exists $\gamma > 0$ such that
\[
\Gamma(h) \ge \gamma \int_0^1 h'(x)^2 + h(x)^2 ~dx \quad \mbox{for all } h \in
\C{H}_0^1,
\]
where $\C{H}_0^1$ is the space of function with $h'$ square integrable
over $[0, 1]$ and $h(0) = h(1) = 0$.
\item[{\rm (C4)}]
Whenever $u$ is a solution of {\rm (\ref{EE})} with $u(0) = u(a) = 0$ for
some $0 < a \le 1$, $u$ is identically zero on $[0, a]$;
thus there are no points $a \in (0, 1]$ conjugate to $x = 0$.
\item[{\rm (C5)}]
The solution of the initial-value problem for {\rm (\ref{EE})}
with $u_0 = 0$ and $u_1 = 1$ is positive on $(0, 1]$.
\end{enumerate}
\end{theorem}

Condition (C5) is easily checked: simply integrate the linear, second-order
initial-value problem (\ref{EE}) forward from the initial condition
$u(0) = 0$ and $u'(0) = 1$.
If the solution is positive on $(0, 1]$, then (C3) holds.

\begin{proof}

\noindent
{(C1) implies (C2):}
Let $u_0 = u(0)$ and $u_1 = u'(0)$ for the given solution of (\ref{EE}).
If the initial-value problem for (\ref{EE}) is written as a first-order system
\[
\left[ \begin{array}{c} u' \\ v' \end{array}\right] = 
\left[ \begin{array}{cc}
0 & 1 \\ Q/P & - P'/P \end{array}\right]
\left[ \begin{array}{c} u \\ v \end{array} \right], \quad
u(0) = u_0, \quad v(0) = u_1,
\]
then the coefficient matrix is continuous under the assumptions of the theorem.
Hence, by the standard existence theory \cite[Sect. 1.5]{TaylorPDE}
for the solution of a linear system of differential equations,
the solution pair $(u, v)$ is $\C{C}^1$, which implies
that $u \in \C{C}^2$ since $u' = v \in \C{C}^1$.
Since $u$ is strictly positive on $[0, 1]$,
it follows that
\begin{equation}\label{u'/u}
w = -\frac{Pu'}{u} \in \C{C}^1 \quad \mbox{on } [0, 1].
\end{equation}
Differentiating $w$ and using (\ref{EE}) to simplify the derivative of the
numerator in (\ref{u'/u}) shows that $w$ satisfies the Ricatti equation
(\ref{R}).
Since $w \in \C{C}^1$, it is clearly bounded uniformly.
The observation that $w = -Pu'/u$ satisfies the Ricatti equation (\ref{R})
appears in \cite[p. 108]{GelfandFomin}.
\smallskip

\noindent
{(C2) implies (C3):}
As seen earlier, when there exists a solution
to the Ricatti equation (\ref{R}), $\Gamma$ reduces to a perfect square
(\ref{S}).
Define $r = h' + wh/P$.
If $h \in \C{H}_0^1$, then $wh/P$ is continuous and
$r \in \C{L}^2$, the space of square integrable functions on $[0, 1]$.
Solving for $h'$ gives
\begin{equation}\label{h'r}
h'(x) = r(x) - \frac{w(x)h(x)}{P(x)} .
\end{equation}
Using an integrating factor and the fact that $h(0) = 0$, we have
\begin{equation}\label{h-solution}
h(x) = \int_0^x \mbox{exp}\left(\int_t^x - \frac{w(s)}{P(s)} ~ ds \right)
r(t) ~ dt.
\end{equation}
Take the $L^2$ norm of (\ref{h-solution}).
Since the factor multiplying $r(t)$ in (\ref{h-solution})
is bounded uniformly in $x$ and $t \in [0, x]$,
there exists a constant $c$, independent of $r$, such that
\begin{equation}\label{h-bound}
\|h\| \le c \|r\|,
\end{equation}
where $\| \cdot\|$ is the $\C{L}^2$ norm.
Combining (\ref{h'r}) and (\ref{h-bound}), $h'$ is also bounded in terms
of $r$.
Choose $c$ larger if necessary to ensure that
\begin{equation}\label{h'-bound}
\|h'\| \le c\|r\|.
\end{equation}
Since $P(x) \ge \alpha$, we obtain
\[
\Gamma(h) = \int_0^1 P(x) r(x)^2 ~ dx 
\ge \alpha \|r\|^2 =
\frac{\alpha}{2} \left( \|r\|^2 + \|r\|^2 \right) \ge
\frac{\alpha}{2c^2} \left( \|h\|^2 + \|h'\|^2 \right) ,
\]
where (\ref{h-bound}) and (\ref{h'-bound}) are used in the last step.
Hence, (C3) holds with $\gamma = \alpha/(2c^2)$.
\smallskip

\noindent
{(C3) implies (C4):}
Let $u$ denote the solution of (\ref{EE}) given in (C4).
Since $u$ is also the solution of initial-value problem for (\ref{EE})
with $u_0 = u(0)$ and $u_1 = u'(0)$, it follows that $u \in \C{C}^2$ on
$[0, a]$.
Extend $u$ as zero on the interval $[a, 1]$.
The extended $u$ lies in $\C{H}_0^1$.
Since $u$ vanishes on $[a, 1]$, we have
\begin{equation}\label{Gamma-a}
\Gamma(u) = \int_0^a P(x)u'(x)^2 + Q(x) u(x)^2 \; dx.
\end{equation}
Integrating by parts yields
\[
\int_0^a P(x)u'(x)^2 \; dx = \int_0^a -\frac{d}{dx}(P(x)u'(x))u(x) \; dx =
\int_0^1 -Q(x)u(x)^2 \; dx,
\]
where the last equality is due to (\ref{EE}).
Combining this with (\ref{Gamma-a}) yields $\Gamma(u) = 0$, which implies that
$u := 0$ by (C3).
\smallskip

\noindent
{(C4) implies (C5):}
If $u$ is the solution of the initial value problem (\ref{IVP}) for (\ref{EE})
with $u_0 = 0$ and $u_1 = 1$, then $u(x) > 0$ for $x > 0$ near zero.
If for some $a \in (0, 1]$, we have $u(a) = 0$, then (C4) is violated.
Hence, $u$ is positive on $(0, a]$.
\smallskip

\noindent
{(C5) implies (C1):}
Let $U_i$ denote the solution of the initial-value problems
for (\ref{EE}) associated with $(u_0, u_1) =$ $(i, 1-i)$, $i = 0, 1$.
Any linear combination of $U_0$ and $U_1$ is a solution of (\ref{EE}).
By (C5), $U_0$ is positive on $(0, 1]$.
Since $U_1(0) = 1$, we will add to $U_0$ a small positive multiple of $U_1$
to achieve a solution of (\ref{EE}) that is strictly positive on the
closed interval $[0, 1]$.
In detail, since $U_1(0) = 1$ and $U_1$ is continuous, we can choose
$0 < \delta \le 1$ such that $U_1(x) \ge 1/2$ on $[0, \delta]$.
By (C5), $U_0$ is strictly positive on any closed interval
contained in $(0, 1]$.
Consequently,
\[
m = \min \{ U_0(x) : \delta \le x \le 1\} > 0.
\]
Finally, observe that
\begin{equation}\label{u(x)}
u(x) = U_0(x) + m \left( \frac{U_1(x)}{2\|U_1(x)\|_{\max}} \right),
\end{equation}
where $\|U_1(x)\|_{\max}$ is the absolute maximum of $U_1$ over the
interval $[0, 1]$, is strictly positive on $[0, 1]$.
On $[0, \delta]$, the first term in (\ref{u(x)}) is nonnegative and
the second term is strictly positive.
On $[\delta, 1]$, $U_0(x) \ge m > 0$ while the second term in
(\ref{u(x)}) is bounded in absolute value by $m/2$.
Hence, $u(x) \ge m/2$ on $[\delta, 1]$, which implies that $u(x)$ is
strictly positive on $[0, 1]$.
\end{proof}

\section{Second-order sufficient optimality conditions} \label{second}
As a consequence of Theorem~\ref{maintheorem}, we have the following
second-order sufficient optimality conditions.
\smallskip

\begin{corollary}\label{second-order}
The following conditions are sufficient for $y^* \in \C{C}_0^1$ to be a
strict local minimizer of $F$ in $(\ref{P})$:
\begin{enumerate}
\item [{\rm (S1)}]
$f \in \C{C}^2$ and $y^*$ satisfies the Euler equation $(\ref{E})$.
\item [{\rm (S2)}]
We define
\begin{eqnarray*}
P(x) &=& \left( \frac{\partial^2}{(\partial z_1)^2} \right)
f\left( z, x \right) \bigg|_{z = ({y^*}'(x),y^*(x)) }, \quad \mbox{and} \\
Q(x) &=& \left( \frac{\partial^2}{(\partial z_2)^2} - \frac{d}{dx}
\frac{\partial^2}{\partial z_1 \partial z_2}\right)
f\left( z, x \right) \bigg|_{z = ({y^*}'(x),y^*(x))} .
\end{eqnarray*}
$P \in \C{C}^1$ with $P$ strictly positive on $[0, 1]$,
and the mixed partial derivative of $f$ contained in $Q$ is $\C{C}^1$.
\item [{\rm (S3)}]
The solution of the initial-value problem $(\ref{IVP})$ for $(\ref{EE})$ with
$u_0 = 0$ and $u_1 = 1$ is positive on $(0, 1]$.
\end{enumerate}
\end{corollary}
\begin{proof}
Expand the integrand of (\ref{P}) in a second-order Taylor expansion around
$(y^*(x), y^*(x))$ using the intermediate value form for the remainder term.
The first derivative term in the expansion vanishes since $y^*$ satisfies
the Euler equation (\ref{E}) by (S1).
After deleting this first derivative term, the expansion has the form
\begin{equation}\label{Fexpand}
F(y^*+h) = F(y^*) + \frac{1}{2} \bar{\Omega}(h),
\end{equation}
where
\[
\bar{\Omega}(h) = \int_0^1 \bar{P}(x)h'(x)^2 + 2 \bar{R}(x) h'(x)h(x)
+ \bar{Q}(x) h(x)^2 \; dx.
\]
Here $\bar{P}$, $\bar{Q}$, and $\bar{R}$ denote the $(1,1)$, $(2,2)$, and
$(1,2)$ elements of the Hessian
\[
\nabla^2_{z} f\left( z, x \right)
\bigg|_{z = (a(x), b(x))},
\]
where
\[
\left[ \begin{array} {c}
a(x) \\
b(x) \end{array} \right] =
(1-\alpha) \left[ \begin{array} {c}
{y^*}'(x) + h'(x) \\
{y^*}(x) + h(x) \end{array} \right] +
\alpha \left[ \begin{array}{c}
{y^*}'(x) \\
{y^*}(x) \end{array} \right] \quad \mbox{for some } \alpha \in [0, 1].
\]
Add and subtract $\Omega(h)$ on the right side of (\ref{Fexpand})
and replace the second $\Omega$ by $\Gamma$ to obtain
\begin{equation}\label{Fexpand2}
F(y^*+h) = F(y^*) + \frac{1}{2} (\bar{\Omega}(h) - \Omega(h)) +
\frac{1}{2} \Gamma(h).
\end{equation}
By (S3) it follows that (C3) is satisfied.
Hence, for $h$ in a sufficiently small neighborhood of zero in the
max-norm for $\C{C}^1$, the $\Gamma(h)$ term in the expansion
(\ref{Fexpand2}) dominates the 
$(\bar{\Omega}(h) - \Omega(h))$ term.
\end{proof}

\section{Conclusions}

Theorem~\ref{maintheorem} unifies the classic second-order sufficient
optimality conditions by both showing their equivalence, and by introducing
a new equivalent condition (C5) that is easy to check; we simply integrate a
second-order linear differential equation over the interval $[0, 1]$.
If the solution is positive on $(0, 1]$, then $y^*$ is a strict local
minimizer within the framework of Corollary~\ref{second-order}.
\bigskip

\noindent
{\bf Data Availability}
Data sharing is not applicable to this article as no datasets were
generated or analyzed during the current study.
\bigskip

\noindent
{\bf Conflict of Interest}
The author has no competing interests to declare that are relevant
to the content of this article.

\bibliographystyle{siam}

\end{document}